\documentclass[11pt,draft,reqno]{amsart}

\usepackage{amsmath,amssymb,amscd,amsfonts,amsthm}
\usepackage{graphics}
\usepackage{dsfont}
\usepackage[all]{xy}
\usepackage{enumerate}

\usepackage{array,amscd,latexsym, verbatim}

{\catcode`\@=11
\gdef\n@te#1#2{\leavevmode\vadjust{%
 {\setbox\z@\hbox to\z@{\strut#1}%
  \setbox\z@\hbox{\raise\dp\strutbox\box\z@}\ht\z@=\z@\dp\z@=\z@%
  #2\box\z@}}}
\gdef\leftnote#1{\n@te{\hss#1\quad}{}}
\gdef\rightnote#1{\n@te{\quad\kern-\leftskip#1\hss}{\moveright\hsize}}
\gdef\?{\FN@\qumark}
\gdef\qumark{\ifx\next"\DN@"##1"{\leftnote{\rm##1}}\else
 \DN@{\leftnote{\rm??}}\fi{\rm??}\next@}}
\DeclareFontFamily{OT1}{wncyr}{\hyphenchar\font45 }
\DeclareFontShape{OT1}{wncyr}{m}{n}{%
   <5> <6> <7> <8> <9> gen * wncyr
   <10> <10.95> <12> <14.4> <17.28> <20.74>  <24.88>wncyr10}{}
\DeclareFontShape{OT1}{wncyr}{m}{it}{%
   <5> <6> <7> <8> <9> gen * wncyi
   <10> <10.95> <12> <14.4> <17.28> <20.74> <24.88> wncyi10}{}
\DeclareFontShape{OT1}{wncyr}{m}{sc}{%
   <5> <6> <7> <8> <9> <10> <10.95> <12> <14.4>
   <17.28> <20.74> <24.88>wncysc10}{}
\DeclareFontShape{OT1}{wncyr}{b}{n}{%
   <5> <6> <7> <8> <9> gen * wncyb
   <10> <10.95> <12> <14.4> <17.28> <20.74> <24.88>wncyb10}{}
\input cyracc.def
\def\rus{\usefont{OT1}{wncyr}{m}{n}\cyracc\fontsize{9}{11pt}\selectfont}

\DeclareMathSizes{9}{9}{7}{5}

\theoremstyle{plain}

\newtheorem{theorem}{Theorem}

\newtheorem{corollary}{Corollary}

\theoremstyle{definition}

\newtheorem{conjecture}{Conjecture}

\newtheorem{definition}{Definition}
\newtheorem{remark}{\it Remark}

\newtheorem{nothing*}[theorem]{}
\newtheorem{subnothing*}[sub]{}

\theoremstyle{remark}

\begin{document}

%%\title[On a conjecture of \'E.\;Ghys]
%%{On a conjecture of \'E.\;Ghys}
\title[Finite subgroups of diffeomorphism groups]
{Finite subgroups of\\ diffeomorphism groups}

\author[Vladimir  L. Popov]{Vladimir  L. Popov${}^*$}
\address{Steklov Mathematical Institute,
Russian Academy of Sciences, Gubkina 8, Moscow 119991,
Russia}
%%\address{$\mbox{\tiny and}$}
%%{\footnotesize and}

\address{National Research University Higher School of Economics, Myasnitskaya
20, Moscow 101000, Russia} \email{popovvl\char`\@mi.ras.ru}

\thanks{
 ${}^*$\,Supported by
 grants {\rus RFFI
11-01-00185-a}, {\rus N{SH}--2998.2014.1}, and the
program {\it Contemporary Problems of Theoretical
Mathematics} of the Russian Academy of Sciences, Branch
of Mathematics.}

\maketitle

\

\vskip -18mm

\

\begin{abstract}
We prove:\;(1) 
 the existence, for every integer $n\geqslant 4$, of a
 {\it noncompact} smooth $n$-dimensional
 topological ma\-ni\-fold whose diffeo\-mor\-phism group
 contains
an isomorphic copy of {\it every}\, finitely presented group; (2)
a finiteness theorem
on finite simple subgroups of diffeomor\-phism groups of {\it compact} smooth topological manifolds.
\end{abstract}

\vskip 2mm

\noindent {\bf 1.}
 In \cite[Sect.\,2]{Po1} was introduced the definition of abstract Jordan group
 %%(Definition \ref{Jd} below)
  and initiated
 exploration of the following two problems on automorphism groups and birational self-map groups of algebraic varieties:
 \vskip 2mm
%% \noindent{\bf Problems.}  Describe algebraic varieties $X$ for which
%% \begin{enumerate}[\hskip 4.2mm \bf (A)]
%% \item
%% the group ${\rm Aut}(X)$ is Jordan;
%% \item
%% the group ${\rm Bir}(X)$ is Jordan.
%% \end{enumerate}
\noindent{\bf Problem A.} Describe algebraic varieties $X$ for which ${\rm Aut}(X)$ is Jordan.

\noindent{\bf Problem B.} The same with ${\rm Aut}(X)$ replaced by ${\rm Bir}(X)$.

\vskip 2mm
%%\noindent
Recall from \cite[Def.\;2.1]{Po1} that Jordaness of a group is defined as follows:

\begin{definition}%%[{{\rm \cite[Def.\;2.1]{Po1}}}]
\label{Jd}
A group $G$ is called a {\it Jordan group}
if there exists a positive integer $d_G$,
%%$J_G$,
depending on $G$ only,
such that every finite subgroup
$F$ of $G$ contains a normal abelian subgroup whose index in $F$ is at most $d_G$.
\end{definition}

Informally Jordaness means that all finite subgroups of $G$ are ``almost'' abelian in the sense that they are extensions of abelian groups by finite groups taken from a finite list.

Since the time of launching the exploration program of
Problems A and B
 a rather extensive information on them has been obtained (see \cite{Po2}), but, for instance, at this writing (October 2013) is still unknown whether there exists $X$ with non-Jordan group ${\rm Aut}(X)$ (see Question\,1 in \cite{Po2}; it is interesting to juxtapose it with Corollary \ref{2} below).

\vskip 2mm

\noindent {\bf 2.}  In this note the counterpart of Problem A is explored, in which algebraic varieties $X$ are replaced by connected smooth topological manifolds $M$, and ${\rm Aut}(X)$ is replaced by ${\rm Diff}(M)$, the diffeomorphism group of $M$.\;It it  shown that the situation for noncompact manifolds is quite different from that for compact ones.

\vskip 2mm
The author is grateful to J-P. Serre and M. Brin for the comments.

\vskip 2mm

\noindent {\bf 3.}  First, we consider the case of noncompact manifolds.

Recall that a group is called finitely presented if it is presented by finitely many generators and relations.
%% of noncompact manifolds.

\begin{theorem}\label{g} For every integer $n\geqslant 4$, 
there exists a simply connected noncompact smooth oriented $n$-dimensional topological manifold\;$C_n$\;such that the group
${\rm Diff}(C_n)$ contains
an isomorphic copy of every
finitely presented group, and  this copy is a discrete transformation group
of $C_n$ acting freely\footnote{i.e., for this action, the stabilizer of every point of $C_n$ is trivial}.
\end{theorem}

\begin{proof} Since for any smooth topological manifolds $M$ and $N$ the group ${\rm Diff}(M)$ is 
a subgroup of ${\rm Diff}(M\times N)$, it suffices to statement for $n=4$.  
%%[Proof of Theorem {\rm \ref{g}}]
As is known (see, e.g.,\;\cite[Thm.\;12.29]{R}) Higman's Embedding Theorem \cite{H} (see also \cite[Thm.\,12.18]{R}) implies the existence of
a universal finitely presented group, that is a finitely presented group  $\mathcal U$ which contains  as a subgroup an isomorphic copy of every finitely presented group.\;In
%%(i.e., a group presented by finitely many generators and relations).\;In
turn, finite presen\-tedness of  $\mathcal U$ implies the existence of a connected compact smooth
oriented four-dimensional ma\-nifold $B$ whose fundamental group is isomorphic to $\mathcal U$ (see,\,e.g.,\,\cite[Thm.\,5.1.1]{CZ}).\;Consider  the universal cover $\widetilde B\to B$.\;Its
deck transformation group
is a subgroup of ${\rm Diff}(\widetilde B)$ isomorphic to the funda\-mental group of $B$, i.e., to the group $\mathcal U$, and this subgroup is a discrete transformation group of $\widetilde B$ acting freely
 (see,\,e.g.,\;\cite[Chap.\,V, Sect.\,8]{Ma1967}). Therefore, one can take $C_4=\widetilde B$. This completes the proof.
\end{proof}

\noindent{\bf 4.} Since every finite group is finitely presented, Theorem \ref{g}
 yields

\begin{corollary} \label{1}  For every integer $n\geqslant 4$,
there exists a simply connected noncompact smooth oriented $n$-dimensional topological manifold\;$C_n$\;such that the group
${\rm Diff}(C_n)$ contains a freely acting on $C_n$
isomorphic copy of every
finite group.
\end{corollary}

If $G$ a Jordan group, then Definition \ref{Jd} implies the existence of a constant $d_G$, depending on $G$ only,
such that the order of every nonabelian finite simple
subgroup of $G$ is at most $d_G$. Since there are nonabelian finite simple groups whose order
is bigger than any given constant (for instance,  the alternating group ${\rm Alt}_n$ is simple for $n\geqslant 5$ and $|{\rm Alt}_n|=n!/2\xrightarrow[n\to\infty]{} {\infty}$), Theorem \ref{g} yields

\begin{corollary}\label{2}
For every integer $n\geqslant 4$,
there exists a simply connected noncompact smooth oriented $n$-dimensional topological manifold\;$C_n$\;such that the group
${\rm Diff}(C_n)$ is non-Jordan.
\end{corollary}

Corollary \ref{2} answers the question raised in the correspondence with
I.\,Mun\-det i Riera \cite{Po3}, \cite{MiR13}. We note that there are connected noncompact manifolds $M$ such that ${\rm Diff}(M)$ is Jordan.\;For instance, if $M=\mathbb R^n$, this is so for $n=1, 2$ \cite{MiR13}, for $n=3$ \cite{MY}, and for $n=4$  \cite{KS},
but for $n\geqslant 5$ the answer is unknown; it would be interesting to find it. 
\vskip 2mm

\noindent{\bf 5.} If one focuses only on {\it finite} subgroups of diffeomorphism groups, it is possible to give
%%somewhat
somewhat more explicit constructions of manifolds $M$ such that ${\rm Diff}(M)$ contains
an isomorphic copy of {\it every} finite group.\;Below are three of them,
%%such constructions,
each is based on the combination of the
idea used in the proof of Theorem\,\ref{g} with an appropriate result from group theory.

%%\eject

\vskip 2mm

\noindent{\it{Construction}} (i): 

Consider the direct product of symmetric groups
\begin{equation*}
G:=\prod_{n\geqslant 2}{\rm Sym}_n.
\end{equation*}
Since Coxeter groups have a solvable word problem \cite{T},
$G$ has a presenta\-tion with a recursively enumerable set of relations.\;By \cite{HNH}, this and countability of  the group $G$ imply that $G$ can be embedded in a
%%recursively presented
group $G^{\rm II}$ that  has two generators and a recursively enumerable set of relations.\;By  Higman's Embedding Theorem, $G^{\rm II}$ is embeddable in a finitely presented group $\widetilde G$.\;Thus
we may identify $G$ with a subgroup of $\widetilde G$.\;As above, finite presentedness of  $\widetilde G$ implies the existence of a connected compact smooth
oriented four-dimensional manifold whose fundamental group is isomorphic to $\widetilde G$.\;Since $G$ is a subgroup of $\widetilde G$, this manifold is covered by a connected smooth
oriented four-dimensional manifold whose fundamental group is iso\-morphic to\;$G$ (see,\,e.g.,\,\cite[Thm.\,10.2]{Ma1967}).\;The deck transformation group of  the universal covering manifold $S$ of the latter manifold is a subgroup of ${\rm Diff}(S)$ isomorphic to $G$ and acting on $S$ properly discontinuously.\;Since every finite group can be embedded in ${\rm Sym}_n$ for a suitable $n$, we conclude that  Corollaries \ref{1} and \ref{2} hold true with $C$ replaced by $S$.

\vskip 2mm

\noindent{\it Construction} (ii): 

Let $\mathcal N$ be the subgroup of the group of all permutations of ${\bf Z}$ generated by the
transposition $\sigma:=(1,2)$ and the ``translation'' $\tau$ defined by the condition $\tau(i)=i+1$
for every $i\in\bf Z$; see \cite[Example 4]{Po2}. Since $\tau^m\sigma\tau^{-m}=(m+1, m+2)$ for every $m$, and the set of transpositions $(1,2), (2,3),\ldots, (n-1,n)$ generates the symmetric group ${\rm Sym}_n$, the group $\mathcal N$ contains an isomorphic copy of every finite group.\;One can explicitly describe a set of defining relations of $\mathcal N$; this set  is recursively enumerable.
Arguing then
as in Const\-ruction (i), one proves the existence of a simply connected noncompact smooth oriented four-dimensional manifold $N$ such that  ${\rm Diff}(N)$ contains
an isomorphic copy of ${\mathcal N}$
acting on $N$ properly discontinuously. Thus  Corollaries \ref{1} and \ref{2} hold true with $C$ replaced by $N$.

\begin{remark} The group ${\rm Sym}_{\infty}$ of all permutations of ${\bf Z}$ that move only finitely many elements is a subgroup of $\mathcal N$.\;Therefore,
 there
is a simply connected noncompact smooth oriented four-dimensional manifold
whose diffeomor\-phism group
contains
an isomorphic copy of ${\rm Sym}_{\infty}$
that acts on
this manifold properly discontinuously. Representation theory of ${\rm Sym}_{\infty}$
is the subject of many publications.\;It would be interesting to explore which part of represen\-tations of  ${\rm Sym}_{\infty}$ is realizable in the cohomology of such a manifold.
\end{remark}

The existence of an embedding of $\mathcal N$ in ${\rm Diff}(N)$ is worth to compare with the following conjecture formulated in
\cite[Sect.\;3.3, Conj.\,1]{Po2}:

\vskip 2mm

\begin{conjecture}\label{Bir} $\mathcal N$ is not embeddable in ${\rm Bir}(X)$ for every irreducible algeb\-raic variety $X$.
\end{conjecture}

Since $\mathcal N$ contains an isomorphic copy of every finite group,
the following conjecture (cf.\;\cite[Sect.\,3.3, Question\,8]{Po2}) implies Conjecture \ref{Bir}:

\begin{conjecture}\label{simple}
For every irreducible algebraic variety $X$, there are only finitely many pairwise nonisomorphic nonabelian simple finite groups
whose isomorphic copy is contained in ${\rm Bir}(X)$.
\end{conjecture}

\noindent {\it Construction} (iii): 

Consider Richard J. Thompson's group $V$ (concerning its definition and properties used below see,\,e.g.\,\cite[\S{6}]{CFP}).\;As
$V$ is finitely presented,
the same argument as in the proof of Theorem \ref{g} yields the existence of a simply connected noncompact smooth oriented four-dimensional manifold\;$T$\,whose diffeomorphism group
${\rm Diff}(T)$ contains
an isomorphic copy of $V$ acting on $T$ properly discontinuously.\;Since $V$ contains a subgroup isomorphic 
to
${\rm Sym}_n$ for every $n$,
the above Corollaries \ref{1} and \ref{2} hold true with $C$ replaced by $T$.

\begin{remark} Since the explicit presentation of $V$ by 4 generators and 14 relations is known (see \cite[Lemma 6.1]{CFP}),\;the proof of \cite[Thm.\,5.1.1]{CZ} yields the explicit way to construct $T$ from the connected sum of four copies of $S^1\times S^3$ successively glueing to it fourteen ``handles'' $D^2\times S^2$.
\end{remark}

\noindent {\bf 6.} Now we consider the case of compact manifolds.

According to \cite{MiR13},
the following conjecture  is attributed
%%by D.\;Fi\-sher
to \'E.\;Ghys (we reformulate it using Definition \ref{Jd}):

\begin{conjecture}\label{cJ}{\it For\;every\;connected\;com\-pact\;smooth\;manifold\;$M$,\;the gro\-up ${\rm Diff}(M)$ is Jordan.}
\end{conjecture}

%%\vskip 2mm

There are several evidences in favor of  Conjecture \ref{cJ}.
For instance, it is true in either of the following cases:
\begin{enumerate}[\hskip 4.2mm \rm(i)]
\item $M$ is oriented and $\dim(M)\leqslant 2$ (see, e.g., \cite[Thm.\,1.3]{MiR10});
\item $\dim(M)=n$ and $M$ admits an unramified covering $\widetilde{M}\to M$ such that $H^1(\widetilde{M}, \bf Z)$ contains the cohomology classes $\alpha_1,\ldots, \alpha_n$ satisfying
    $\alpha_1\cup \ldots\cup  \alpha_n\neq 0$ (see \cite[Thm.\,1.4(1)]{MiR10}).
     \item $M$ has torsion free integral cohomology supported in even degrees (see\;\cite{MR13Odd}).
\end{enumerate}

%%Given any integer $d$, the classification of finite simple groups and formulas for their orders %%(see,\,e.g.,\,\cite{W})
%%imply finiteness of every set of pairwise nonisomorphic finite simple groups of order $<d$.
%%This and
Definition \ref{Jd} implies that in every Jordan group every set of pairwise nonisomorphic nonabelian finite simple subgroups is finite.\;Theorem \ref{fin} below (based on the classification of finite simple groups) shows that a certain finiteness property of nonabelian finite simple subgroups indeed holds for diffeomorphism groups of connected compact manifolds.\;So Theorem \ref{fin}
may be considered as another evidence in favor of Conjecture \ref{cJ}.\;It shows that
noncompactness of the above-considered ``highly symmetric'' manifolds $C$, $S$, $N$, and $T$ admitting a free action of every finite group is an inalienable property  (the peculiarity here is considering  actions of all finite groups on the {\it same} manifold: it is known \cite{A} that every finite group acts freely on some {\it compact} two-dimensional manifold, which, however, depends on this group).

\vskip 2mm

\noindent {\bf 7.} In order to formulate Theorem \ref{fin}, recall  (see, e.g., \cite[1.2]{W}) that by the adverted classification
%%theorem
the complete list of nonabelian finite simple groups (considered up to isomorphism) consists
of 26 sporadic groups and the follo\-wing infinite series depending on parameters (below we use the notation
of \cite{W}):

\begin{enumerate}[\hskip 4mm $\bullet$]
\item series $\mathcal S_1$ depending on one parameter $n\in \mathbb N:=\{1, 2, \ldots\}$:
\begin{align}\label{A}
{\rm Alt}_n, n\geqslant 5;
{}^2\!B_2(2^{2n+1}); {}^2\!G_2(3^{2n+1}); {}^2\!F_4(2^{2n+1});
\end{align}
\item series $\mathcal S_2$ depending on two parameters $p, a\in \mathbb N$, where $p$ is a prime:
\begin{align}\label{E}
G_2(p^a), (p, a)\neq (2, 1); F_4(p^a); E_6(p^a); {}^2\!E_6(p^a); {}^3\!D_4(p^a); E_7(p^a); E_8(p^a);
\end{align}
\item series ${\mathcal S}_3$ depending on three parameters $p, a, n\in \mathbb N$, where $p$ is a prime:
\begin{equation}\label{C}
\begin{split}
&{\rm PSL}_n(p^a), n\geqslant 2, (n, p, a)\neq (2, 2, 1), (2, 3, 1);\\
&{\rm PSU}_n(p^a), n\geqslant 3, (n, p, a)\neq (3, 2, 1);\\
&{\rm PSp}_{2n}(p^a), n\geqslant 2, (n, p, a)\neq (2, 2, 1);\\
&{\rm P}\Omega_{2n+1}(p^a), n\geqslant 3, p\neq 2;\\
&{\rm P}\Omega^+_{2n}(p^a), n\geqslant 4;\\
&{\rm P}\Omega^-_{2n}(p^a), n\geqslant 4.
\end{split}
\end{equation}
\end{enumerate}

\begin{theorem}\label{fin}
Let $M$ be a connected compact smooth manifold.\,There exists a real number $b_M$, depending on $M$ only, such that
every nonabelian finite simple subgroup $F$ of $\;{\rm Diff}(M)$, belonging to one of the series
$\mathcal S_1$, $\mathcal S_2$, or $\mathcal S_3$, has
the following property:
\begin{enumerate}[\hskip 4mm \rm(i)]
\item if $F\in\mathcal S_1$, then $n\leqslant b_M$;
\item if $F\in \mathcal S_2$, then $a\leqslant b_M$;
\item if $F\in\mathcal S_3$, then $a\leqslant b_M$ and $n\leqslant b_M$.
\end{enumerate}
\end{theorem}

\begin{proof}
 1. The general plan is as follows.  By \cite[Thm.\,2.5]{MS}, 
 there exists a real number $e^{\ }_M$,  depending on $M$
only, such that the rank of every ele\-mentary Abelian subgroup\footnote{Recall that a finite Abelian group $A$ is called ele\-mentary if the order of every nonidentity element of $A$ is equal to a prime number $p$ (depending on $A$). The order of $A$ is then $p^r$, and the integer $r$ is called the rank of $A$.} of $\,{\rm Diff}(M)$ is at most $e_M$.\;This inequality is then applied to the suitable  elementary Abelian subgroups of the groups of each type listed in \eqref{A}, \eqref{E}, and \eqref{C}; finding these elementary subgroups is performed case by case. This yields the upper bounds of the parameters $n$ and/or $a$ corresponding to each of the types listed in
\eqref{A}, \eqref{E}, and \eqref{C}.\;The maximum of these upper bounds in then the sought-for bound $b_M$.
 
 Following this plan, we now consider separately every group $F$ from the lists
 \eqref{A}, \eqref{E} è \eqref{C}.

 2. Let $F\!=\!{\rm Alt}_n$, $n\geqslant 5$.\;Let $d$\,be the quotient of dividing
  $n$ by\;$3$.\;The subgroup of
${\rm Alt}_n$ generated by the $3$-cycles $(1, 2, 3), \ldots, (3d-2, 3d-1, 3d)$  is an elementary Abelian group of order  $3^d$.\;Therefore, $d\leqslant e^{\ }_M$; whence  $n\leqslant 3e^{\ }_M+2$.

 3. Let $F={}^2\!B_2(2^{2n+1})$. Then $F$ contains an elementary Abelian subgroup of order
 $2^{2n+1}$ (see\;\cite[Sect.\;4.2.2, p.\,115]{W}). Therefore, $2n+1\leqslant e^{\ }_M$, whence $n\leqslant (e^{\ }_M-1)/2$.

 4. Let $F={}^2\!G_2(3^{2n+1})$. Then $F$ contains an elementary Abelian subgroup of order
$3^{6n+3}$
(see\;\cite[Thm.\;4.2(i)]{W}). Therefore, $6n+3\leqslant e^{\ }_M$; whence $n\leqslant (e^{\ }_M-3)/6$.

5. Let $F={}^2\!F_4(2^{2n+1})$. Then $F$ contains an elementary Abelian subgroup of order
$2^{10n+5}$ (see\;\cite[\S3]{Vd}). Therefore, $10n+5\leqslant e^{\ }_M$; whence $n\leqslant (e^{\ }_M-5)/10$.

6. Let $F=G_2(p^a), (p, a)\neq (2, 1)$.
Then $F$ containes an elementary Abelian subgroup of order $p^{3a}$
(see\;\cite[\S3]{Vd}). Therefore, $3a \leqslant e^{\ }_M$; whence
$a\leqslant e^{\ }_M/3$.

6. Let $F=F_4(p^a)$. Then $F$ contains an elementary Abelian subgroup of order
$p^{9a}$ (see\;\cite[\S3]{Vd}). Therefore, $9a \leqslant e^{\ }_M$; whence
$a\leqslant e^{\ }_M/9$.

7. Let $F=E_6(p^a)$.  Then $F$ contains an elementary Abelian subgroup of order
$p^{16a}$ (see\;\cite[\S3]{Vd}).  Therefore, $16a \leqslant e^{\ }_M$; whence
$a\leqslant e^{\ }_M/16$.

8. Let $F={}^2\!E_6(p^a)$.  Then $F$ contains an elementary Abelian subgroup of order
$p^{6a}$ (see\;\cite[\S3]{Vd}).  Therefore, $6a \leqslant e^{\ }_M$; whence
$a\leqslant e^{\ }_M/6$.

9.   Let $F={}^3\!D_4(p^a)$.\;Then $F$ contains an elementary Abelian subgroup of order $p^{9a}$ (see\;\cite[Thm.\,4.3(i)]{W}). Therefore, $9a\leqslant e^{\ }_M$; whence $a\leqslant e^{\ }_M/9$.

10. Let $F=E_7(p^a)$.\;Then $F$ contains an elementary Abelian subgroup of order
$p^{27a}$ (see\;\cite[\S3]{Vd}).
Therefore, $27a \leqslant e^{\ }_M$; whence
$a\leqslant e^{\ }_M/27$.

11. Let $F=E_8(p^a)$.\;Then $F$ contains an elementary Abelian subgroup of order
$p^{36a}$ (see\;\cite[\S3]{Vd}).
Therefore, $36a \leqslant e^{\ }_M$; whence
$a\leqslant e^{\ }_M/36$.

12.  Let $F={\rm PSL}_n(p^a), n\geqslant 2, (n, p, a)\neq (2, 2, 1), (2, 3, 1)$.\;For every positi\-ve integer $s<n$,
 $F$ contains an elementary Abelian subgroup of order $p^{as(n-s)}$  (see\;\cite[Sect.\;3.3.3]{W} and \cite[Thm.\;2.1]{Ba79}). Therefore, $a[n/2](n- [n/2])\leqslant e^{\ }_M$.

13. Let $F\!=\!{\rm PSU}_n(p^a), n\!\geqslant\! 3, (n, p, a)\!\neq\! (3, 2, 1)$.\;Then $F$ contains an ele\-mentary Abelian subgroup of order $p^{a(n-1)^2/4}$\;(see\;\cite[Thm.\;1 and the last paragraph of  \S3]{Wo83}, and also \cite[Sect.\;3.6.2]{W}).\;Therefore, $a(n-1)^2\!/4\!\leqslant\! e^{\ }_M$.

14.  Let $F={\rm PSp}_{2n}(p^a), n\!\geqslant\! 2, (n, p, a)\!\neq \!(2, 2, 1)$.
Then $F$ contains an elementary Abelian subgroup of order $p^{an(n+1)/2}$
(see\;\cite[Thms. 3.7(i), 3.8(i)]{W} and \cite[Thm.\;2.5, Cor.\;4.3]{Ba79}). Therefore, $an(n+1)/2\leqslant e^{\ }_M$.

15.  Let $F={\rm P}\Omega_{2n+1}(p^a), n\geqslant 3, p\neq 2$.
Then $F$ contains an elementary Abelian subgroup of order $p^{an(n-1)/2}$
(see\;\cite[Thm.\;3.10(i)]{W} and \cite[Thms.\;4.1, 4.2, 5.1, 5.2, 5.3]{Ba79}). Therefore,
$an(n-1)/2\leqslant e^{\ }_M$.

16.  Let $F={\rm P}\Omega^+_{2n}(p^a), n\geqslant 4$.
Then $F$ contains an elementary Abelian subgroup of order $p^{an(n-1)/2}$
(see\;\cite[Thm.\;3.12(i)]{W} and \cite[Thms.\;3.1, 3.2]{Ba79}). Therefore,
$an(n-1)/2 \leqslant e^{\ }_M$.

17.  Let $F={\rm P}\Omega^-_{2n}(p^a), n\geqslant 4$.
Then $F$ contains an elementary Abelian subgroup of order $p^{a(n-1)(n-2)/2}$
(see\;\cite[Thm.\;3.11(i)]{W}). Therefore, $a(n-1)(n-2)/2 \leqslant e^{\ }_M$.

This completes the proof.
\end{proof}


\begin{thebibliography}{MiR2010}
 
 

\bibitem[A57]{A}  R.\;D.\;Anderson, {\it Zero-dimensional compact groups of homeomorphisms}, Pacif. J. Math. {\bf 7} (1957), 797--810.

 \bibitem[Ba79]{Ba79}    M.\;J.\;Barry, {\it Large Abelian subgroups of Chevalley groups}, J. Aust. Math. Soc., Ser. A {\bf 27} (1979), no. 1,
59--87.


    \bibitem [CFP96]{CFP} J.\;W.\;Cannon, W.\;J.\;Floyd, W.\;R.\;Parry, {\it Introductory notes on Richard Thompson's groups}, L'Enseignement Math. Revue Internat., IIe S\'er. {\bf 42} (1996), no.\;3, 215--256.

\bibitem[CZ93]{CZ}  D.\;J.\;Collins, H.\;Zieschang, {\it Combinatorial group theory and fundamental gro\-ups},
 in: {\it Algebra {\rm VII}. Combinatorial Groups Theory, Applications to Geo\-met\-ry}, Encyclopaedia of Mathemtical Sciences, Vol. 58, Springer-Verlag, Berlin, 1993, pp.\;3--165.


     \bibitem[HNH61]{H}   G.\;Higman, {\it Subgroups of finitely presented groups}, Proc. Roy. Soc. {\262} (1961), 455--475.

\bibitem[HNH49]{HNH}   G.\;Higman, B.\;H.\;Neumann, H.\;Neumann, {\it Embedding theorem for groups}, J. London Math. Soc. {\bf 24} (1949) , 247--254.


\bibitem[KS?]{KS} S.\;Kwasik, R.\;Schultz, {\it Finite symmetries of $\;{\mathbb R}^4$ and $S^4$}, preprint.


\bibitem[MS63]{MS}    L.\;N.\;Mann, J.\;C.\;Su, {\it Actions of elementary $p$-groups on manifolds}, Trans.
Amer. Math. Soc. {\bf 106} (1963), 115--126.



\bibitem[Ma67]{Ma1967} W.\;S.\;Massey, {\it Algebraic Topology: An Introduction},
Graduate Texts in Mathe\-matics, Vol. 56, Springer-Verlag, New York, 1977.


\bibitem[MiR10]{MiR10} I.\;Mundet i Riera, {\it Jordan's theorem for the diffeomorphism group of some
manifolds}, Proc. Amer. Math. Soc. {\bf 138} (2010), no. 6, 2253--2262.



    \bibitem[MiR$13_1$]{MiR13} I.\;Mundet i Riera, {\it Letters  to V.\;L.\;Popov}, July $30$ and August 1, $2013$.


         \bibitem[MiR$13_2$]{MR13Odd} I.\;Mundet i Riera, {\it Finite group actions on manifolds without odd cohomology}, arXiv:1310.6565v2.



             \bibitem [MY84]{MY}    W.\;H\;Meeks,  Sh.-T.\;Yau, {\it Group actions on ${\mathbb R}^3$}, in:
{\it The Smith Conjecture} (New York, 1979),
Pure Appl. Math., Vol.\;112, Academic Press, Orlando, FL, 1984, pp.\;167--179.



\bibitem[Po10]{Po1}   V.\;L.\;Popov, {\it On the Makar-Limanov,
 Derksen invariants, and finite auto\-mor\-phism groups  of algebra\-ic va\-ri\-eti\-es}, in:\;{\it Affine Algebraic Geometry: The Rus\-sell Festschrift},
 CRM Proceedings and Lecture Notes, Vol.\;54, Amer. Math. Soc., 2011, pp.\,289--311.


  \bibitem[Po$13_1$]{Po3} V.\;L.\;Popov, {\it Letter to I.\;Mundet i Riera}, July 31, 2013.

 \bibitem[Po$13_2$]{Po2}   V.\;L.\;Popov, {\it  Jordan groups and automorphism groups of algebraic varieties}, to appear in: {\it Groups of Automorphisms in Birational and Affine Geometry} (Levico Terme, Italy, October 29--November 3, 2012), Lecture Notes in Mathematics, 2014, Springer, Heidelberg,
{\tt arXiv:1307.5522v4}.

\bibitem[R10]{R} J.\;J.\;Rotman, {\it An Introduction to the Theory of Groups}, 4th ed., Graduate Texts in Mathematics, Springer-Verlag, New York, 1995.


\bibitem[T69]{T}    J.\;Tits, {\it Le probl\'em des mots dans les groupes de Coxeter}, in: {\it Symposia
Mathematica} (INDAM, Rome, 1967/68), Vol. 1, Academic
Press, London, 1969, pp.\;175--185.

 \bibitem[Vd01]{Vd}  E.\;P.\;Vdovin, {\it Large Abelian unipotent subgroups of finite Chevalley groups}, Algebra and Logic {\bf 40} (2001), no. 5, 292--305.

\bibitem[W09]{W}  R.\;A.\;Wilson, {\it The Finite Simple Groups}, Graduate Texts in Mathematics, Vol. 251, Springer, London, 2009.



\bibitem[Wo83]{Wo83}    W.\;J.\;Wong, {\it Abelian unipotent subgroups of finite unitary and symplectic groups}, J. Aust. Math.
Soc., Ser. A {\bf 33} (1983), no. 3, 331--344.



 


\end{thebibliography}
 \end{document}